\documentclass[11pt,letterpaper]{article}
\usepackage[T1]{fontenc}
\usepackage{lmodern}
\usepackage{amsmath,amssymb,amsthm,mathtools}
\usepackage[margin=1in]{geometry}
\usepackage{microtype,titlesec}
\usepackage[colorlinks=true,linkcolor=black,citecolor=black,urlcolor=black]{hyperref}
\hypersetup{pdfauthor={Jingchen Hu, Heming Jiao},pdftitle={A Liouville Theorem for the Real k-Hessian Equation with a One-Sided Growth Condition},pdfsubject={Interior Hessian estimates, quadratic lower growth, and Liouville rigidity}}
\newtheoremstyle{plainupright}{7pt}{7pt}{\normalfont}{}{\bfseries}{.}{.5em}{}
\theoremstyle{plainupright}
\newtheorem{theorem}{Theorem}[section]
\newtheorem{lemma}[theorem]{Lemma}
\newtheorem{proposition}[theorem]{Proposition}

\newtheorem{remark}[theorem]{Remark}
\numberwithin{equation}{section}
\newcommand{\R}{\mathbb R}
\newcommand{\G}{\Gamma}
\newcommand{\clG}{\overline{\Gamma}}
\newcommand{\eps}{\varepsilon}
\newcommand{\tr}{\operatorname{tr}}
\newcommand{\diam}{\operatorname{diam}}
\newcommand{\dist}{\operatorname{dist}}
\title{A Liouville Theorem\\
for the Real $k$-Hessian Equation\\
with a One-Sided Growth Condition}
\author{Jingchen Hu \and Heming Jiao}
\date{October 25, 2026}
\begin{document}
\maketitle
\setlength{\parskip}{2pt}
\setlength{\abovedisplayskip}{6pt plus 2pt minus 1pt}
\setlength{\belowdisplayskip}{6pt plus 2pt minus 1pt}
\setlength{\abovedisplayshortskip}{4pt plus 1pt}
\setlength{\belowdisplayshortskip}{5pt plus 1pt}

\begin{abstract}
We establish a scale-invariant interior $C^2$ estimate for smooth solutions of the real $k$-Hessian equation. As an application, we prove that every smooth entire solution of $\sigma_k(D^2u)=1$ satisfying $u(x)\ge b|x|^2-B$ for some constants $b,B$ with $b>0$ is a quadratic polynomial. For smooth solutions, admissibility is forced by the growth condition.

\smallskip The results also hold for viscosity solutions and for weak solutions in Trudinger's sense. Admissibility is required in the latter case.
\end{abstract}

\section{Introduction}

We study the rigidity of entire solutions to the real
$k$-Hessian equation
\begin{equation}\label{eq:pde}
\sigma_k(D^2u)=1\qquad\text{in }\R^n,\qquad 2\le k\le n.
\end{equation}
Here $\sigma_j$ denotes the $j$-th elementary symmetric polynomial,
evaluated at the eigenvalues of the real Hessian $D^2u$.
We use $\sigma_0=1$ and $\sigma_j=0$ when $j$ exceeds the number of variables.

The operators $\sigma_k$ connect the Laplacian and the real
Monge--Amp\`ere operator. Their elliptic branches are described by
the cones
\begin{equation}\label{eq:cone}
\G_j^{(m)}=\{\lambda\in\R^m:\sigma_\ell(\lambda)>0
\text{ for }1\le\ell\le j\},\qquad 1\le j\le m.
\end{equation}
We write $H\in\G_j$ when the eigenvalues of a symmetric matrix $H$
belong to this cone. For smooth solutions of \eqref{eq:pde}, we use
$k$-admissible to mean $D^2u\in\G_k$ at every point.

The weak admissibility conventions are specified in
Sections~\ref{sec:viscosity-formulation} and~\ref{sec:ibp-formulation}.

A natural hypothesis is positive quadratic lower growth:
\begin{equation}\label{eq:lower}
u(x)\ge b|x|^2-B\qquad\text{for some constants }b>0\text{ and }B\ge0.
\end{equation}
This condition makes large sublevel sets bounded. It does not directly
bound the oscillation of a rescaled solution on a fixed Euclidean ball,
which is needed in the usual interior-estimate argument.

\subsection{Earlier estimates and rigidity theorems}

Liouville theorems for nonlinear partial differential equations have been
studied extensively. For Hessian equations, a central question is quadratic
rigidity under growth or convexity assumptions.

Bao, Chen, Guan, and Ji proved quadratic rigidity for strictly convex
entire solutions under \eqref{eq:lower} \cite[Theorem 3.2]{BCGJ}.
For $2\le k\le n-1$, Li, Ren, and Wang extended this result to
$(k+1)$-admissible solutions using a Pogorelov estimate \cite{LRW}.

Chu and Dinew treated a general class of Hessian equations through
weighted gradient and Hessian estimates. For $k$-admissible solutions
of \eqref{eq:pde}, their result applies under \eqref{eq:lower} and the additional bound
$\sigma_{k+1}(D^2u)\ge-A$ \cite[Corollary 1.3]{CD}.

Within the natural class of $k$-admissible solutions, Wang and Bao
proved rigidity under positive two-sided quadratic growth
\cite[Theorem 1.1]{WB}. Their proof combines interior estimates with
rescaling and H\"older continuity of the Hessian.

Chen and Xiang proved that every $2$-admissible solution
$u\in C^4(\R^3)$ of $\sigma_2(D^2u)=1$ satisfying \eqref{eq:lower}
is a quadratic polynomial \cite[Theorem 1.4]{CX}.

Bin Wang proved a lower-growth rigidity theorem for $k$-admissible
solutions when $2k>n$ \cite[Theorem 1.2]{BW}. His argument establishes
a Pogorelov estimate independent of the gradient bound.

For the complex Monge--Amp\`ere equation $\det(v_{i\bar j})=1$,
Chu and Liu proved quadratic rigidity for smooth entire plurisubharmonic
solutions under positive two-sided quadratic growth \cite[Theorem 1.2]{CL}.
At the same time, using a different method, Hu and Sheng established
interior $C^2$ estimates and proved quadratic
rigidity for locally bounded entire plurisubharmonic Bedford--Taylor solutions
under weak quadratic growth conditions \cite[Theorems 1.1 and 1.3]{HS}.
In particular, the Liouville theorem in that paper permits a positive
quadratic lower bound together with an upper bound of order
$(1+|z|^2)\log(e+|z|^2)$. We expect that this upper growth bound
can be further weakened.

Motivated by these developments, particularly the recent work of Wang
\cite{BW}, we ask whether the upper growth bound can be removed entirely
for the real $k$-Hessian equation, so that \eqref{eq:lower} alone implies
quadratic rigidity.

\subsection{The main results}

We first state the Liouville theorem, followed by the interior estimate
used in its proof.

\begin{theorem}[Liouville theorem]\label{thm:main}
Let $2\le k\le n$, and suppose $u\in C^\infty(\R^n)$ satisfies
\begin{equation}\label{eq:main}
\sigma_k(D^2u)=1\qquad\text{in }\R^n.
\end{equation}
Suppose there exist constants $b>0$ and $B\ge0$ such that
\begin{equation}\label{eq:liouville-growth}
u(x)\ge b|x|^2-B.
\end{equation}
Then
\begin{equation}\label{eq:conclusion}
u(x)=\tfrac12x^TAx+\ell\cdot x+c,
\qquad\sigma_k(A)=1,\qquad A\ge2bI,
\end{equation}
where $A$ is a constant symmetric matrix, $\ell\in\R^n$, and $c\in\R$.
\end{theorem}

The following estimate is essential to the proof of the Liouville theorem.
Its constants depend only on the dimension, the order of the equation,
and the coefficient of the quadratic lower bound.

\begin{theorem}[Interior $C^2$ estimate]\label{thm:interior}
Let $2\le k\le n$ and $b>0$. There exist constants
$\rho=\rho(n,k,b)\in(0,1)$ and $C=C(n,k,b)>0$ with the following property.

For $R>0$, suppose $u\in C^\infty(B_R)$ satisfies
\begin{equation}\label{eq:local-equation}
\sigma_k(D^2u)=1\quad\text{in }B_R,\qquad u(0)=0,
\end{equation}
and, for some constant $B$,
\begin{equation}\label{eq:local-growth}
u(x)\ge b|x|^2-B\quad(x\in B_R),\qquad
0\le B\le\frac{bR^2}{4}.
\end{equation}
Then
\begin{equation}\label{eq:interior-estimate}
\sup_{B_{\rho R}}\left(\frac{|u|}{R^2}+
\frac{|Du|}{R}+|D^2u|\right)\le C.
\end{equation}
\end{theorem}

The Liouville theorem and the interior estimate extend to viscosity
solutions, as stated in the following two theorems.

\begin{theorem}[Liouville theorem for viscosity solutions]\label{thm:viscosity-liouville}
Let $2\le k\le n$, and let $u\in C(\R^n)$ be a viscosity solution,
in the sense of Section~\ref{sec:viscosity-formulation}, of
\begin{equation}\label{eq:viscosity-entire-equation}
\sigma_k(D^2u)=1\qquad\text{in }\R^n.
\end{equation}
Suppose that, for some constants $b>0$ and $B\ge0$,
\begin{equation}\label{eq:viscosity-growth}
u(x)\ge b|x|^2-B.
\end{equation}
Then
\begin{equation}\label{eq:viscosity-conclusion}
u(x)=\tfrac12x^TAx+\ell\cdot x+c,
\qquad\sigma_k(A)=1,\qquad A\ge2bI,
\end{equation}
where $A$ is a constant symmetric matrix, $\ell\in\R^n$, and $c\in\R$.
\end{theorem}

\begin{theorem}[Interior estimate for viscosity solutions]\label{thm:viscosity-interior}
Let $2\le k\le n$ and $b>0$. There exist constants
$\rho=\rho(n,k,b)\in(0,1)$ and $C=C(n,k,b)>0$ with the following property.

For $R>0$, suppose $u\in C(B_R)$ is a viscosity solution, in the sense
of Section~\ref{sec:viscosity-formulation}, of
\begin{equation}\label{eq:viscosity-local-equation}
\sigma_k(D^2u)=1\quad\text{in }B_R,\qquad u(0)=0.
\end{equation}
Assume that, for some constant $B$,
\begin{equation}\label{eq:viscosity-local-growth}
u(x)\ge b|x|^2-B\quad(x\in B_R),\qquad
0\le B\le\frac{bR^2}{4}.
\end{equation}
Then $u\in C^\infty(B_{\rho R})$ and
\begin{equation}\label{eq:viscosity-interior-estimate}
\sup_{B_{\rho R}}\left(\frac{|u|}{R^2}+
\frac{|Du|}{R}+|D^2u|\right)\le C.
\end{equation}
\end{theorem}

\begin{remark}[Pogorelov's example and the growth condition]\label{rem:pogorelov}
Pogorelov's example shows that viscosity solutions can have an unbounded
Hessian when the growth hypothesis is absent; see \cite[Section 4.4]{Mooney}.
We use the construction in $k$ variables and extend it independently of
the remaining $n-k$ variables. For $3\le k\le n$, write
\begin{equation}\label{eq:pogorelov-coordinates}
x=(y,t,z)\in\R^{k-1}\times\R\times\R^{n-k},\qquad
\alpha_k=2-\frac2k.
\end{equation}
Choose a positive smooth solution $h$ on $(-T,T)$, for some $T>0$, of
\begin{equation}\label{eq:pogorelov-ode}
\alpha_k^{k-1}h^{k-2}
\bigl((\alpha_k-1)hh''-\alpha_k(h')^2\bigr)=1,
\qquad h(0)=1,\quad h'(0)=0.
\end{equation}
The function
\begin{equation}\label{eq:pogorelov-function}
P(y,t,z)=|y|^{\alpha_k}h(t)
\end{equation}
is defined on the slab $\R^{k-1}\times(-T,T)\times\R^{n-k}$.
For $r=|y|>0$, differentiation gives
\[
\det D^2_{y,t}P
=\alpha_k^{k-1}h^{k-2}
\bigl((\alpha_k-1)hh''-\alpha_k(h')^2\bigr)
r^{k\alpha_k-2k+2}=1.
\]
Here $k\alpha_k-2k+2=0$. The $k-2$ angular eigenvalues are positive,
and the radial--$t$ block has positive radial entry and determinant.
Thus the active $k\times k$ Hessian block is positive definite.

The remaining $n-k$ eigenvalues vanish, so $\sigma_k(D^2P)=1$ off
$\{y=0\}$. Since $\alpha_k>1$, the gradient extends continuously across
this set; restriction to lines therefore shows that $P$ is convex.

We verify the viscosity equation on $\{y=0\}$, where $P$ is not $C^2$.

At $y=0$, no $C^2$ upper touching test exists because $1<\alpha_k<2$.
A lower touching test has Hessian $A$ with $A|_S\le0$ on the flat subspace
$S=\{y=0\}$. If $A\in\clG_k$ and $\sigma_k(A)>0$, Maclaurin's inequalities
give $A\in\G_k$.

Restrict successively to $k-1$ hyperplanes within the preceding subspace,
each containing $S$.
Equation~\eqref{eq:restriction} gives $A|_S\in\G_1^{(n-k+1)}$, so
$\tr(A|_S)>0$, contradicting $A|_S\le0$.

Thus every admissible lower test has $\sigma_k(A)=0$.
This verifies the supersolution inequality on the singular set.

The growth condition in Theorem~\ref{thm:viscosity-interior} excludes this
example. Indeed, take $0<R<T$ and $x_*=(0,3R/4,0)\in B_R$.
Since $P(0)=P(x_*)=0$, the assumed lower bound would imply
\begin{equation}\label{eq:pogorelov-growth-contradiction}
0=P(x_*)\ge\frac{9bR^2}{16}-B
\ge\frac{5bR^2}{16}>0.
\end{equation}
Thus the quantitative lower bound prevents the flat singular axis of
Pogorelov's example. This example is defined on a slab and concerns local
regularity.
\end{remark}

The results also extend to admissible weak solutions in the sense of
Trudinger \cite{Tr97}. We recall his definition in
Section~\ref{sec:ibp-formulation}.

\begin{theorem}[Liouville theorem for admissible weak solutions]\label{thm:ibp-liouville}
Let $2\le k\le n$, and let $u\in C(\R^n)$ be an admissible weak solution,
in Trudinger's sense stated in Section~\ref{sec:ibp-formulation}, of
\begin{equation}\label{eq:ibp-entire-equation}
\sigma_k(D^2u)=1\qquad\text{in }\R^n.
\end{equation}
Suppose that, for some constants $b>0$ and $B\ge0$,
\begin{equation}\label{eq:ibp-growth}
u(x)\ge b|x|^2-B.
\end{equation}
Then
\begin{equation}\label{eq:ibp-conclusion}
u(x)=\tfrac12x^TAx+\ell\cdot x+c,
\qquad\sigma_k(A)=1,\qquad A\ge2bI,
\end{equation}
where $A$ is a constant symmetric matrix, $\ell\in\R^n$, and $c\in\R$.
\end{theorem}

\begin{theorem}[Interior estimate for admissible weak solutions]\label{thm:ibp-interior}
Let $2\le k\le n$ and $b>0$. There exist constants
$\rho=\rho(n,k,b)\in(0,1)$ and $C=C(n,k,b)>0$ with the following property.

For $R>0$, suppose $u\in C(B_R)$ is an admissible weak solution,
in Trudinger's sense stated in Section~\ref{sec:ibp-formulation}, of
\begin{equation}\label{eq:ibp-local-equation}
\sigma_k(D^2u)=1\quad\text{in }B_R,\qquad u(0)=0.
\end{equation}
Assume that, for some constant $B$,
\begin{equation}\label{eq:ibp-local-growth}
u(x)\ge b|x|^2-B\quad(x\in B_R),\qquad
0\le B\le\frac{bR^2}{4}.
\end{equation}
Then $u\in C^\infty(B_{\rho R})$ and
\begin{equation}\label{eq:ibp-interior-estimate}
\sup_{B_{\rho R}}\left(\frac{|u|}{R^2}+
\frac{|Du|}{R}+|D^2u|\right)\le C.
\end{equation}
\end{theorem}

\begin{remark}[Distributional solutions without admissibility]\label{rem:weak-counterexample}
The conclusions of Theorems~\ref{thm:ibp-liouville} and~\ref{thm:ibp-interior}
fail for unrestricted distributional solutions of the $2$-Hessian equation.
We give an example in every dimension $n\ge2$.

For $w\in W^{1,2}_{\mathrm{loc}}$, the distribution
\begin{equation}\label{eq:very-weak-sigma-two}
\mathcal S_2[w]
=\frac12\sum_{i,j}\partial_{ij}(w_iw_j)-\frac12\Delta|Dw|^2
\end{equation}
agrees with $\sigma_2(D^2w)$ for smooth $w$; see
\cite[Definition 1.1 and Section 2.1]{LQ}, with our normalization.
Its pairing with a test function $\phi$ is given by the gradient
integrals on the right-hand side of \eqref{eq:ibp-two-gradient}.

Set
\begin{equation}\label{eq:folded-quadratic}
c=\binom{n}{2}^{-1/2},\qquad
q(x)=\frac c2|x|^2+cx_1,\qquad w(x)=|q(x)|.
\end{equation}
Since $Dw=\operatorname{sgn}(q)Dq$ almost everywhere, we obtain
\begin{equation}\label{eq:folded-equation}
Dw\otimes Dw=Dq\otimes Dq,\qquad
\mathcal S_2[w]=\mathcal S_2[q]=\binom{n}{2}c^2=1.
\end{equation}
Moreover, completing the square gives
\begin{equation}\label{eq:folded-growth}
w(x)\ge q(x)\ge\frac c4|x|^2-c.
\end{equation}
Thus $w$ has positive quadratic lower growth but is not quadratic.
On $\{q<0\}$, its Hessian is $-cI$, so $w$ is not $2$-admissible.

The conclusion of Theorem~\ref{thm:ibp-interior} also fails for this
distributional class: $w(0)=0$, and the choices
$b=c/4$, $B=c$, and $R=4$ satisfy $B=bR^2/4$.
Nevertheless, $w(te_1)=|ct+ct^2/2|$ is not differentiable at the origin.
\end{remark}

The upper-test inequality itself implies admissibility in
Theorem~\ref{thm:viscosity-liouville} and Theorem~\ref{thm:viscosity-interior};
see Lemma~\ref{lem:viscosity-admissibility}.
In Theorem~\ref{thm:ibp-liouville} and Theorem~\ref{thm:ibp-interior},
admissibility is part of Trudinger's weak-solution definition.

For the smooth solution in Theorem~\ref{thm:interior}, the growth bound
ensures that $u$ attains a minimum inside the ball.
At this point its Hessian is positive semidefinite, and the equation
gives admissibility.

Maclaurin's inequalities then give positive lower bounds for the lower
symmetric polynomials. These bounds prevent the Hessian from reaching
the boundary of $\G_k$; see Section~\ref{sec:branch}.

To prove Theorem~\ref{thm:interior}, we choose a regular sublevel
component whose diameter is bounded by a multiple of $R$.
A distance barrier shows that this component contains a ball of
radius comparable to $R$.

The same barrier separates every fixed negative sublevel from the
boundary. The interior gradient estimate of Chou and Wang then
provides the gradient bound needed in their weighted Hessian estimate
\cite[Theorems 3.2 and 1.5]{CW}.

The construction has two geometric features. When $k<n$, the principal
curvatures of a $(k-1)$-convex boundary may be negative. The algebraic
inequality in Lemma~\ref{lem:transform} controls their fractional transformation.

The distance function is also nonsmooth at the cut locus.
Local semiconvexity of the distance power passes its almost-everywhere
Hessian inequalities to smooth convolutions, which can be used in the
comparison argument.

\subsection{Organization of the paper}

Section~\ref{sec:distance} constructs the admissible distance power,
and Section~\ref{sec:boundary} proves the uniform boundary estimate.
Section~\ref{sec:interior} proves Theorem~\ref{thm:interior}.

Section~\ref{sec:liouville} applies this estimate and the Evans--Krylov
theorem to prove Theorem~\ref{thm:main}.
Section~\ref{sec:viscosity} proves the viscosity results,
Theorems~\ref{thm:viscosity-liouville} and~\ref{thm:viscosity-interior}.
Section~\ref{sec:ibp} treats admissible weak solutions and proves
Theorems~\ref{thm:ibp-liouville} and~\ref{thm:ibp-interior}.
Appendix~\ref{sec:algebra} contains the two algebraic lemmas.

\section{Admissible powers of the distance function}\label{sec:distance}

We construct a function that vanishes on the boundary and whose Hessian
belongs to the closed admissible cone. At smooth points of the distance
function, this follows from Lemma~\ref{lem:transform}.

For $2\le k\le n$, fix
\begin{equation}\label{eq:parameters}
N=n-k+1,\qquad
\theta=\frac2k\left(1-\frac1{2N}\right),\qquad
\alpha=1-\theta,\qquad\beta=\frac\alpha k.
\end{equation}
These constants satisfy $0<\theta<1$ and $0<\beta<\alpha<1$.

Let $\Omega\subset\R^n$ be a nonempty bounded smooth domain whose
outward principal curvatures belong to $\G_{k-1}^{(n-1)}$.
We take the outward principal curvatures to be positive on spheres.

Set
\begin{equation}\label{eq:distance-def}
D=\diam\Omega,\qquad d(x)=\dist(x,\partial\Omega),\qquad v(x)=-d(x)^\alpha.
\end{equation}
A function $f$ is locally semiconvex if, on a convex neighborhood of
each point, $f(x)+\tfrac C2|x|^2$ is convex for some $C\ge0$.
Local semiconcavity means that $-f$ is locally semiconvex.

The spectral matrix cone $\G_j$ is convex and is preserved under
addition of positive-semidefinite matrices. The same statements hold
for $\clG_j$. These standard properties are summarized in
\cite[Section 2]{BW}.

At smooth points of $d$, the normal eigenvalue of $D^2v$ is positive.
Lemma~\ref{lem:transform}
shows that it compensates for the possible negative tangential
eigenvalues, with a choice of $\alpha$ depending only on $n$ and $k$.

\begin{lemma}[The distance power]\label{lem:distance}
The function $v=-d^\alpha$ is locally semiconvex. At almost every point,
\begin{equation}\label{eq:distance-hessian}
D^2v\in\clG_k,\qquad
\Delta v\ge c_0:=\alpha(1-\alpha)D^{\alpha-2}.
\end{equation}
Its convolutions with nonnegative mollifiers satisfy the same inequalities
wherever the convolution is defined inside $\Omega$.
\end{lemma}

\begin{proof}
The distance $d$ is Lipschitz and hence differentiable almost everywhere.
If $y$ is a nearest boundary point to such an $x$, then $d(z)\le|z-y|$
with equality at $x$. Differentiation gives
\[
Dd(x)=\frac{x-y}{d(x)},\qquad y=x-d(x)Dd(x).
\]
Thus the nearest boundary point is unique. The focal values of the smooth
normal map have measure zero by Sard's theorem, so it remains to compute
at points with a unique, nonfocal nearest boundary point.

Let $\nu$ be the outward unit normal. The minimizing segment is normal
and points inward, so $x=y-d(x)\nu(y)$.

Write $\Psi(y,s)=y-s\nu(y)$, and let $e_i$ be a principal direction at $y$.
With our sign convention for the curvatures,
\[
D_y\Psi(y,s)e_i=(1-s\kappa_i(y))e_i.
\]
Choose a unit-speed boundary curve $\gamma_i$ through $y$ with
$\gamma_i'(0)=e_i$. Since $y$ minimizes the distance to $x$, our curvature
convention gives
\[
0\le\left.\frac{d^2}{dt^2}\frac{|x-\gamma_i(t)|^2}{2}\right|_{t=0}
=1-d\kappa_i(y).
\]
Nonfocality makes these inequalities strict, so the normal map has
a smooth local inverse.

Let $x_m\to x$ and choose any nearest boundary points $y_m$.
Every convergent subsequence of $y_m$ has a nearest point to $x$ as its
limit, because distance is continuous.

Uniqueness of the nearest point and compactness give $y_m\to y$.
Also $d(x_m)\to d(x)$, so the pairs $(y_m,d(x_m))$ remain in the
local inverse neighborhood for all sufficiently large $m$.

In this neighborhood, the inverse normal coordinate $s$ equals $d$.
Thus $d$ is smooth there, with
\[
Dd(\Psi(y,s))=-\nu(y).
\]
Differentiating this local gradient identity in each principal direction
gives the eigenvalues of $D^2d$:
\begin{equation}\label{eq:distance-eigenvalues}
-\frac{\kappa_1(y)}{1-d\kappa_1(y)},\ldots,
-\frac{\kappa_{n-1}(y)}{1-d\kappa_{n-1}(y)},\ 0,
\qquad 1-d\kappa_i(y)>0.
\end{equation}
To compute the Hessian of $v$, apply the chain rule:
\[
D^2v=-\alpha d^{\alpha-1}D^2d
+\alpha(1-\alpha)d^{\alpha-2}Dd\otimes Dd.
\]
The second term acts only in the normal direction, since $Dd$ is a unit
normal. Introduce $a_i=d\kappa_i(y)$ and $\mu_i=a_i/(1-a_i)$.
The eigenvalues of $D^2v$ are therefore
\begin{equation}\label{eq:power-eigenvalues}
\alpha d^{\alpha-2}(\mu_1,\ldots,\mu_{n-1},1-\alpha).
\end{equation}
Here $a\in\G_{k-1}$ and $a_i<1$. Since $1-\alpha=\theta$,
Lemma~\ref{lem:transform} gives $D^2v\in\clG_k$.

The same lemma gives $\sigma_1(\mu)>0$. Using $d\le D$ and
$\alpha-2<0$, we obtain
\begin{equation}\label{eq:trace-lower}
\Delta v
=\alpha d^{\alpha-2}\bigl(\sigma_1(\mu)+1-\alpha\bigr)
\ge\alpha(1-\alpha)D^{\alpha-2}=c_0.
\end{equation}

The distance $d$ is locally semiconcave in $\Omega$; see
\cite[Proposition 2.2.2]{CS}. We can check the needed semiconvexity of $v$
directly on any convex neighborhood $V\Subset\Omega$.

Choose $a>0$ such that $\dist(V,\partial\Omega)\ge a$.
For $y\in\partial\Omega$, write $r=|x-y|$ and $e=(x-y)/r$.
Since $r\ge a$, differentiation gives
\[
D^2_x(-|x-y|^\alpha)
=\alpha r^{\alpha-2}\bigl((2-\alpha)e\otimes e-I\bigr)
\ge-\alpha a^{\alpha-2}I.
\]
Hence each function $-|x-y|^\alpha+\tfrac12\alpha a^{\alpha-2}|x|^2$
is convex on $V$. Their supremum is convex as well, and
\[
v(x)=\sup_{y\in\partial\Omega}(-|x-y|^\alpha).
\]
This proves local semiconvexity, including at points where the nearest
boundary point is not unique.

The distributional Hessian of the locally semiconvex function $v$ is a
matrix-valued Radon measure; see \cite[Theorem 2.3.1]{CS}.
Its Lebesgue decomposition is
\begin{equation}\label{eq:hessian-measure}
D^2v=H\,dx+M_s,\qquad
H\in\clG_k\ \text{a.e.},\qquad\tr H\ge c_0.
\end{equation}
The density $H$ agrees almost everywhere with the classical Hessian
computed above. Locally, $v+\tfrac C2|x|^2$ is convex for some $C\ge0$,
so $D^2v+CI\,dx$ is a positive matrix-valued measure.

The added quadratic changes only the absolutely continuous part.
Hence the singular part $M_s$ is positive semidefinite.

Let $\varphi_\rho$ be a nonnegative mollifier of integral one.
For $v_\rho=v*\varphi_\rho$, distributional differentiation gives
\begin{equation}\label{eq:mollification}
D^2v_\rho=H*\varphi_\rho+M_s*\varphi_\rho.
\end{equation}
The first term is a nonnegative average of matrices in $\clG_k$,
so it belongs to $\clG_k$ by convexity. The second term is positive
semidefinite and therefore preserves this cone under addition.

Taking the trace in \eqref{eq:mollification} also gives
\[
\Delta v_\rho=(\tr H)*\varphi_\rho+(\tr M_s)*\varphi_\rho
\ge c_0\int\varphi_\rho=c_0.
\]
These arguments apply whenever the support of the translated mollifier
is contained in $\Omega$.
\end{proof}

\section{A uniform boundary estimate}\label{sec:boundary}

We perturb the distance power by a quadratic function. The perturbation
makes the barrier strictly admissible, while its trace lower bound
supplies the strict differential inequality used in comparison.

\begin{proposition}[Uniform boundary estimate]\label{prop:boundary}
Let $\Omega\subset\R^n$ be a nonempty bounded smooth domain whose
outward principal curvatures belong to $\G_{k-1}^{(n-1)}$.
Suppose $U\in C^2(\Omega)\cap C(\overline\Omega)$ is $k$-admissible and
\begin{equation}\label{eq:dirichlet}
\sigma_k(D^2U)=1\quad\text{in }\Omega,
\qquad U=0\quad\text{on }\partial\Omega.
\end{equation}
Set $D=\diam\Omega$ and $d(x)=\dist(x,\partial\Omega)$.
There is a constant $C_0=C_0(n,k)>0$ such that
\begin{equation}\label{eq:boundary}
0\le-U(x)\le C_0D^2\left(\frac{d(x)}D\right)^\beta
\qquad(x\in\Omega),
\end{equation}
where $\beta$ is defined in \eqref{eq:parameters}.
\end{proposition}

The constant is independent of the boundary curvature and its
smoothness norms. We prove the proposition in the following two steps.

Since $D^2U\in\G_k$, we have $\Delta U>0$. The maximum principle and
the boundary condition give $U\le0$.

\subsection{Prepare the comparison function}

Fix $x_0\in\Omega$ and $\eps>0$. Introduce the barrier
\begin{equation}\label{eq:barrier}
w=Kv+\eps\bigl(|x-x_0|^2-D^2\bigr),\qquad
K=L_{n,k}D^{2-\alpha}\eps^{1-k},
\end{equation}
where
\begin{equation}\label{eq:Lconstant}
L_{n,k}
=\frac{2}{2^{k-1}\binom{n-1}{k-1}\alpha(1-\alpha)}.
\end{equation}
On $\partial\Omega$, we have $v=0$ and $|x-x_0|\le D$.
Consequently, $w\le0=U$ there.

For a symmetric matrix $A\in\clG_k$, expansion of the symmetric
polynomial gives
\begin{equation}\label{eq:expansion}
\sigma_k(KA+2\eps I)
=\sum_{j=0}^k\binom{n-j}{k-j}(2\eps)^{k-j}K^j\sigma_j(A).
\end{equation}
To see the coefficient, choose $j$ eigenvalue factors from $KA$ in each
product defining $\sigma_k$; the remaining $k-j$ factors contribute
$2\eps$ and can be chosen in $\binom{n-j}{k-j}$ ways.

Since $A\in\clG_k$, every $\sigma_j(A)$ in this expansion is
nonnegative. Keeping the $j=1$ term gives
\begin{equation}\label{eq:linear-term}
\sigma_k(KA+2\eps I)
\ge\binom{n-1}{k-1}(2\eps)^{k-1}K\tr A.
\end{equation}

Apply Lemma~\ref{lem:distance} to $v_\rho$, and set
$w_\rho=Kv_\rho+\eps(|x-x_0|^2-D^2)$.
Adding $2\eps I$ makes all the symmetric polynomials of orders
$1,\ldots,k$ strictly positive, by the same expansion.

The powers of $D$ and $\eps$ in $K$ cancel those in
$(2\eps)^{k-1}c_0$. Thus \eqref{eq:distance-hessian} and
\eqref{eq:Lconstant} give the fixed strict inequality
\begin{equation}\label{eq:strict-barrier}
D^2w_\rho\in\G_k,\qquad
\sigma_k(D^2w_\rho)
\ge\binom{n-1}{k-1}(2\eps)^{k-1}Kc_0=2.
\end{equation}

\subsection{Exclude the interior maximum}

We prove $w\le U$. Suppose
$M:=\max_{\overline\Omega}(w-U)>0$.
Continuity gives $\delta>0$ such that $w-U<M/4$ whenever $d\le2\delta$.

The compact set $K_\delta=\{x\in\Omega:d(x)\ge\delta\}$ contains
a point $x_M$ where $w-U=M$. Every point of $\partial K_\delta$
satisfies $d=\delta$, so the collar estimate applies there.

Choose $\rho<\delta/2$ so small that
$\|w_\rho-w\|_{L^\infty(K_\delta)}<M/4$.

At $x_M$ we have $w_\rho-U>3M/4$, whereas on $\partial K_\delta$
we have $w_\rho-U<M/2$. Therefore its maximum on $K_\delta$ occurs
at an interior point, where the mollified barrier is smooth.

At that point, $D^2w_\rho\le D^2U$. Monotonicity of $\sigma_k$ on
$\G_k$ and \eqref{eq:strict-barrier} give
\begin{equation}\label{eq:comparison}
2\le\sigma_k(D^2w_\rho)\le\sigma_k(D^2U)=1.
\end{equation}
This contradiction excludes a positive maximum and proves $w\le U$.

Evaluating at $x_0$ yields
\begin{equation}\label{eq:preoptimization}
-U(x_0)\le L_{n,k}D^{2-\alpha}\eps^{1-k}d(x_0)^\alpha
+\eps D^2.
\end{equation}
We balance the two terms by choosing
$\eps=L_{n,k}^{1/k}(d(x_0)/D)^{\alpha/k}$.
Since $\beta=\alpha/k$, this choice satisfies
\[
L_{n,k}D^{2-\alpha}\eps^{1-k}d(x_0)^\alpha
=\eps D^2
=L_{n,k}^{1/k}D^2\left(\frac{d(x_0)}D\right)^\beta.
\]
Substitution into \eqref{eq:preoptimization} proves \eqref{eq:boundary}
with $C_0=2L_{n,k}^{1/k}$. The point $x_0$ was arbitrary.
\qed

\section{Proof of the interior estimate}\label{sec:interior}

We first convert Proposition~\ref{prop:boundary} into an interior
estimate for the zero-boundary problem. We then apply that estimate
to a sublevel component of the solution in Theorem~\ref{thm:interior}.

\subsection{Estimate the solution away from the boundary}

\begin{lemma}[Interior estimate on a zero-boundary domain]\label{lem:domain-interior}
Let $\Omega$ satisfy the assumptions of Proposition~\ref{prop:boundary},
and let $U\in C^\infty(\Omega)\cap C(\overline\Omega)$ be an admissible
solution of \eqref{eq:dirichlet}. For $D=\diam\Omega$ and $0<\delta<1$,
\begin{equation}\label{eq:domain-interior}
\sup_{\{d\ge\delta D\}}
\left(\frac{|U|}{D^2}+\frac{|DU|}{D}+|D^2U|\right)
\le C(n,k,\delta),
\end{equation}
whenever the set in the supremum is nonempty.
\end{lemma}

\begin{proof}
Translation and the rescaling $U(Dy)/D^2$ reduce the proof to $D=1$.
The Hessian is unchanged by this scaling, while the gradient is divided
by $D$ and the function by $D^2$.

Proposition~\ref{prop:boundary} then gives
\begin{equation}\label{eq:normalized-boundary}
0\le-U\le C_0d^\beta,\qquad \|U\|_{L^\infty(\Omega)}\le C_0.
\end{equation}
Fix $\eta>0$. If $-U(x)\ge\eta/4$, the boundary estimate forces
$C_0d(x)^\beta\ge\eta/4$. Solving for $d(x)$ gives
\begin{equation}\label{eq:sublevel-separation}
-U(x)\ge\frac\eta4
\quad\Longrightarrow\quad
d(x)\ge r_\eta:=\left(\frac{\eta}{4C_0}\right)^{1/\beta}.
\end{equation}

For the constant right-hand side, all derivatives vanish, so
condition (3.13) of \cite{CW} holds with $h=0$.
Apply their interior gradient estimate \cite[Theorem 3.2]{CW}
on balls of radius $r_\eta/2$ centered in this sublevel.

Together with \eqref{eq:normalized-boundary}, it gives
\begin{equation}\label{eq:sublevel-gradient}
\sup_{\{-U\ge\eta/4\}}|DU|\le C(n,k,\eta).
\end{equation}
The balls lie inside $\Omega$ by \eqref{eq:sublevel-separation}.
Their radius and the oscillation of $U$ depend only on $n,k,\eta$,
which explains the uniform constant in \eqref{eq:sublevel-gradient}.

If the sublevel is empty, the corresponding estimate is vacuous.

Choose $s\in(\eta/4,\eta/2)$ such that $-s$ is a regular value of $U$,
and put
\begin{equation}\label{eq:inner-dirichlet}
\Omega_s=\{U<-s\},\qquad V=U+s.
\end{equation}
The set $\overline{\Omega_s}$ lies in $\{-U\ge\eta/4\}$, so
\eqref{eq:sublevel-separation} places it a positive distance from
$\partial\Omega$. Regularity of the chosen value makes its boundary smooth.

Adding $s$ leaves the Hessian unchanged. Thus $V$ is admissible,
solves the same equation, and vanishes on $\partial\Omega_s$.

The regular level is $(k-1)$-convex, by the tangential Hessian identity
\eqref{eq:restriction} below. Equations \eqref{eq:normalized-boundary}
and \eqref{eq:sublevel-gradient} bound the $C^1$ norm of $V$.

Since $\overline{\Omega_s}\Subset\Omega$, the function $V$ is smooth
on a neighborhood of $\overline{\Omega_s}$. The right-hand side is
identically one, so the remaining hypotheses of the estimate are satisfied.

Apply the zero-boundary estimate of Chou and Wang
\cite[Theorem 1.5]{CW} on each component of $\Omega_s$:
\begin{equation}\label{eq:pogorelov}
(-V)^4|D^2V|\le C(n,k,\eta).
\end{equation}
The constant in their estimate is independent of the domain.
For a point with $-U\ge\eta$, the choice $s<\eta/2$ gives
\[
-V=-U-s\ge\eta-s>\frac\eta2.
\]
Dividing \eqref{eq:pogorelov} by this positive fourth power, and using
$D^2V=D^2U$, gives
\begin{equation}\label{eq:deep-hessian}
\sup_{\{-U\ge\eta\}}|D^2U|\le C(n,k,\eta).
\end{equation}

To express this bound in terms of distance, fix $x\in\Omega$ and
$0<r<d(x)$. On $B_r(x)$ introduce
\begin{equation}\label{eq:interior-quadratic}
q(y)=a_{n,k}(|y-x|^2-r^2),\qquad
 a_{n,k}=\frac1{2\binom nk^{1/k}}.
\end{equation}
All eigenvalues of $D^2q$ equal $2a_{n,k}>0$, and hence
\[
\sigma_k(D^2q)=\binom nk(2a_{n,k})^k=1.
\]
On $\partial B_r(x)$, we have $U\le0=q$.
Both functions are admissible solutions of the same equation, so the
comparison principle gives $U\le q$ throughout the ball.

Evaluate at $x$ and let $r\uparrow d(x)$. This yields
\begin{equation}\label{eq:depth-distance}
-U(x)\ge a_{n,k}d(x)^2.
\end{equation}
Taking $\eta=a_{n,k}\delta^2$ in \eqref{eq:deep-hessian} bounds the
Hessian on $\{d\ge\delta\}$.

The bound for $U$ follows from \eqref{eq:normalized-boundary}.
The gradient bound follows from \eqref{eq:sublevel-gradient}, because
\eqref{eq:depth-distance} places $\{d\ge\delta\}$ inside $\{-U\ge\eta\}$.

Scaling back proves \eqref{eq:domain-interior}.
\end{proof}

\subsection{Select the admissible branch}\label{sec:branch}

Let $u$ satisfy the hypotheses of Theorem~\ref{thm:interior}.
Introduce the rescaled function $u_R(y)=R^{-2}u(Ry)$ on $B_1$.
Its derivatives satisfy
\[
Du_R(y)=R^{-1}Du(Ry),\qquad D^2u_R(y)=D^2u(Ry).
\]
Thus the equation is unchanged, and the lower bound becomes
$u_R(y)\ge b|y|^2-B/R^2$. The hypothesis on $B$ gives
$0\le B/R^2\le b/4$.

It therefore suffices to prove \eqref{eq:interior-estimate} when $R=1$.
In this normalization, $u(0)=0$, $u(x)\ge b|x|^2-B$, and $0\le B\le b/4$.

The growth bound gives
\begin{equation}\label{eq:branch-minimum}
u(0)=0,\qquad
\liminf_{|x|\to1}u(x)\ge b-B\ge\frac{3b}{4}>0.
\end{equation}
Choose $r_0\in(1/2,1)$. The quadratic lower bound is positive on
$B_1\setminus\overline B_{r_0}$, whereas $u(0)=0$.
A minimum on $\overline B_{r_0}$ is therefore also a minimum in $B_1$.

At such a point $x_*\in B_1$, the second derivative test gives
\begin{equation}\label{eq:positive-point}
D^2u(x_*)\ge0.
\end{equation}
Since $\sigma_k(D^2u(x_*))=1$, at least $k$ of its nonnegative
eigenvalues are positive. Hence $\sigma_j(D^2u(x_*))>0$ for
$1\le j\le k$, so $D^2u(x_*)\in\G_k$.

We now use a continuity argument. Fix $x\in B_1$ and follow the
Hessian along the segment from $x_*$ to $x$:
\begin{equation}\label{eq:hessian-path}
H(t)=D^2u\bigl((1-t)x_*+tx\bigr),\qquad 0\le t\le1.
\end{equation}
Define
\begin{equation}\label{eq:continuity-set}
G=\{t\in[0,1]:H(t)\in\G_k\}.
\end{equation}
Since $H(0)\in\G_k$, we have $0\in G$.
If $t_0\in G$, then every number $\sigma_j(H(t_0))$, $1\le j\le k$,
is positive. Continuity preserves these finitely many strict inequalities
for $t$ sufficiently close to $t_0$, so $G$ is relatively open.

To prove closedness, we use Maclaurin's inequalities: for $\lambda\in\G_k$,
\begin{equation}\label{eq:maclaurin}
\left(\frac{\sigma_{s-1}(\lambda)}{\binom{n}{s-1}}\right)^{1/(s-1)}
\ge
\left(\frac{\sigma_s(\lambda)}{\binom{n}{s}}\right)^{1/s},
\qquad 2\le s\le k.
\end{equation}
Iterating these inequalities and using $\sigma_k(H(t))=1$, we obtain
\begin{equation}\label{eq:admissible-lower}
\sigma_j(H(t))\ge
\binom{n}{j}\binom{n}{k}^{-j/k}>0,
\qquad t\in G,\quad 1\le j<k.
\end{equation}
The lower bounds in \eqref{eq:admissible-lower} are independent of $t$.
Let $t_m\in G$ converge to $t\in[0,1]$. Continuity gives
\[
\sigma_j(H(t))=\lim_{m\to\infty}\sigma_j(H(t_m))
\ge\binom nj\binom nk^{-j/k}>0,\qquad 1\le j<k.
\]
The equation also gives $\sigma_k(H(t))=1$.
Thus all defining inequalities of $\G_k$ remain strict at $H(t)$,
which proves $t\in G$ and the relative closedness of $G$.

Since $[0,1]$ is connected and $G$ is nonempty, open, and closed in
$[0,1]$, we conclude that $G=[0,1]$. In particular,
$D^2u(x)=H(1)\in\G_k$.

As $x\in B_1$ was arbitrary, $u$ is admissible throughout the ball.
We can now apply admissible comparison and the interior estimates.

\subsection{Describe the regular sublevel}

We first explain the curvature of an admissible level set.
For a symmetric matrix $H$ and $0\le j\le n-1$, its $j$-th Newton
tensor is the symmetric matrix $T_j(H)$ defined by
\[
D\sigma_{j+1}(H)[K]=\sum_{p,q=1}^n T_j^{pq}(H)K_{pq},\qquad K=K^T.
\]

If $H\in\G_k$ and $E=\nu^\perp$ for a unit vector $\nu$, then
\begin{equation}\label{eq:restriction}
\sigma_j(H|_E)=\langle T_j(H)\nu,\nu\rangle>0,
\qquad 1\le j\le k-1.
\end{equation}
To see the equality, choose orthonormal coordinates with last vector
$\nu$. Differentiating $\sigma_{j+1}(H)$ with respect to $H_{nn}$
leaves the sum of all $j$-dimensional principal minors of $H|_E$.

Since $H\in\G_k\subset\G_{j+1}$, ellipticity makes $T_j(H)$
positive definite. Evaluating its quadratic form on $\nu$ proves
positivity in \eqref{eq:restriction}.

By Sard's theorem, choose a regular value $t\in(b/4,b/2)$.
Since $u(0)=0<t$, the set $\{u<t\}$ has a component $\Omega_t$
containing the origin.

For $x\in\Omega_t$, the lower bound gives $b|x|^2-B\le u(x)<t$.
Using $t<b/2$ and $B\le b/4$, we obtain
\begin{equation}\label{eq:diameter-sublevel}
\Omega_t\subset B_{\sqrt{(t+B)/b}}\subset B_{\sqrt3/2},
\qquad D_t:=\diam\Omega_t\le\sqrt3.
\end{equation}
Its boundary is smooth and compactly contained in $B_1$.
The outward normal and second fundamental form are
\begin{equation}\label{eq:second-fundamental}
\nu=\frac{Du}{|Du|},\qquad
\mathrm{II}_{\partial\Omega_t}
=\frac{D^2u|_{T\partial\Omega_t}}{|Du|}.
\end{equation}
The second identity follows by differentiating $\nu=Du/|Du|$
in tangent directions. The term from differentiating $|Du|^{-1}$
vanishes when paired with a tangent vector.

Equation \eqref{eq:restriction} places the tangential Hessian in
$\G_{k-1}$. Division by the positive number $|Du|$ preserves this cone,
so the principal curvatures belong to $\G_{k-1}$.

The function $U=u-t$ is admissible, satisfies $\sigma_k(D^2U)=1$,
and vanishes on $\partial\Omega_t$. It therefore satisfies the hypotheses of
Proposition~\ref{prop:boundary} and Lemma~\ref{lem:domain-interior}.

\subsection{Apply the domain estimate}

Set $d_t(x)=\dist(x,\partial\Omega_t)$. At the origin, the boundary
estimate and \eqref{eq:diameter-sublevel} imply
\begin{equation}\label{eq:depth-sublevel}
\frac b4<t=-U(0)
\le C_0D_t^{2-\beta}d_t(0)^\beta
\le C_0\,3^{1-\beta/2}d_t(0)^\beta.
\end{equation}
Define constants depending only on $n,k,b$ by
\begin{equation}\label{eq:rho-definition}
\gamma=\min\left\{\frac14,
\left(\frac{b}{4C_0\,3^{1-\beta/2}}\right)^{1/\beta}\right\},
\qquad \rho=\frac\gamma2.
\end{equation}
Then $d_t(0)\ge\gamma$, so $B_\gamma\subset\Omega_t$.
Indeed, a segment from the origin to a point outside $\Omega_t$ must
first meet $\partial\Omega_t$, at distance at least $d_t(0)$.

Distance to a fixed closed set is $1$-Lipschitz.
Consequently, for $x\in B_\rho$ we have
\begin{equation}\label{eq:relative-interior}
d_t(x)\ge d_t(0)-|x|\ge\rho\ge\frac\rho{\sqrt3}D_t.
\end{equation}

Apply Lemma~\ref{lem:domain-interior} with $\delta=\rho/\sqrt3$.
Since $D_t\le\sqrt3$ and $u=U+t$ with $t<b/2$, we obtain
\begin{equation}\label{eq:unit-interior}
\sup_{B_\rho}\bigl(|u|+|Du|+|D^2u|\bigr)\le C(n,k,b).
\end{equation}
Apply \eqref{eq:unit-interior} to $u_R$ and use the derivative identities
from Section~\ref{sec:branch}. The three terms become
$R^{-2}|u|$, $R^{-1}|Du|$, and $|D^2u|$ on $B_{\rho R}$,
which proves Theorem~\ref{thm:interior}.
\qed

\section{Proof of the Liouville theorem}\label{sec:liouville}

We apply Theorem~\ref{thm:interior} on expanding balls. The resulting
Hessian bound makes the equation uniformly elliptic. The Evans--Krylov
estimate then forces the Hessian to be constant.

\subsection{Obtain a global Hessian bound}

Let $u$ satisfy the hypotheses of Theorem~\ref{thm:main}. Normalize its
value at the origin by setting
\begin{equation}\label{eq:global-normalization}
v=u-u(0),\qquad \widetilde B=B+u(0)\ge0.
\end{equation}
The last inequality follows from \eqref{eq:liouville-growth} at $x=0$.
For every sufficiently large $R$, we have
\begin{equation}\label{eq:normalized-global-growth}
v(x)\ge b|x|^2-\widetilde B\quad(x\in B_R),\qquad
0\le\widetilde B\le\frac{bR^2}{4}.
\end{equation}
Theorem~\ref{thm:interior} therefore gives
\begin{equation}\label{eq:global-hessian}
\sup_{B_{\rho R}}|D^2u|\le C(n,k,b),\qquad
\sup_{\R^n}|D^2u|\le M:=C(n,k,b).
\end{equation}
For each fixed $x\in\R^n$, choose $R$ large enough that
$\widetilde B\le bR^2/4$ and $|x|<\rho R$.
The first estimate then bounds $|D^2u(x)|$ by the same constant $M$.

The branch argument in Section~\ref{sec:branch} applies on each of
these balls. Since the balls exhaust $\R^n$, it also gives
$D^2u\in\G_k$ everywhere.

\subsection{Rescale the uniformly elliptic equation}

Set $F(H)=\sigma_k(H)^{1/k}$ and introduce the set containing all
Hessian values:
\begin{equation}\label{eq:compact-range}
\mathcal K=\{H\in\clG_k:\ \sigma_k(H)=1,\ |H|\le M\}.
\end{equation}
The Maclaurin inequalities \eqref{eq:maclaurin} extend by continuity
to $\clG_k$. Since $\sigma_k(H)=1$ on $\mathcal K$, they give
$\sigma_j(H)>0$ for every $H\in\mathcal K$ and $1\le j<k$.

Thus $\mathcal K$ is a compact subset of $\G_k$.
The derivative matrix is
\[
F^{ij}(H)=\frac1k\sigma_k(H)^{1/k-1}T_{k-1}^{ij}(H).
\]
It is positive definite on $\G_k$ and continuous on $\mathcal K$.
Compactness therefore gives constants $0<\lambda\le\Lambda$, depending
only on $n,k,M$, such that $\lambda I\le(F^{ij}(H))\le\Lambda I$.

For completeness, extend the operator to all symmetric matrices by
\begin{equation}\label{eq:operator-extension}
\widetilde F(H)=\inf_{A\in\mathcal K}
\left\{1+\sum_{i,j}F^{ij}(A)(H_{ij}-A_{ij})\right\}.
\end{equation}
The matrices $A$ and $(F^{ij}(A))$ are bounded on $\mathcal K$,
so this infimum is finite for every $H$.
An infimum of affine functions is concave.

For every positive-semidefinite matrix $P$, each affine function in
\eqref{eq:operator-extension} increases by a number between
$\lambda\tr P$ and $\Lambda\tr P$.

Taking infima preserves these bounds, so
\[
\lambda\tr P\le\widetilde F(H+P)-\widetilde F(H)
\le\Lambda\tr P,\qquad P\ge0.
\]
Thus $\widetilde F$ is uniformly elliptic on all symmetric matrices.
For $H,A\in\mathcal K$, concavity of $F$ gives
\[
1=F(H)\le F(A)+\sum_{i,j}F^{ij}(A)(H_{ij}-A_{ij}).
\]
Every expression in the infimum is therefore at least one.
Choosing $A=H$ attains one, which proves
$\widetilde F(H)=1$ on $\mathcal K$.

Subtract the affine Taylor polynomial at the origin and rescale:
\begin{equation}\label{eq:affine-rescaling}
w_R(y)=\frac{u(Ry)-u(0)-RDu(0)\cdot y}{R^2},
\qquad y\in B_1.
\end{equation}
The normalization gives $w_R(0)=0$ and $Dw_R(0)=0$.
Also, $D^2w_R(y)=D^2u(Ry)$, so \eqref{eq:global-hessian} bounds
its Hessian independently of $R$.

Integrating along the segment from $0$ to $y$ yields
\[
|Dw_R(y)|\le C(n)M|y|,\qquad
|w_R(y)|\le C(n)M|y|^2.
\]
Hence $\|w_R\|_{C^{1,1}(B_1)}\le C(n)M$.
Its Hessian takes values in $\mathcal K$, and therefore
$\widetilde F(D^2w_R)=1$.

To use the convex formulation of \cite[Theorem 1]{EK}, consider
$-w_R$ and the operator $H\mapsto1-\widetilde F(-H)$.
This operator is convex and uniformly elliptic.

Its value on $D^2(-w_R)$ is zero. The cited Evans--Krylov estimate
therefore gives constants $\zeta\in(0,1)$ and $C>0$, independent of $R$,
such that
\begin{equation}\label{eq:ek-rescaled}
[D^2w_R]_{C^\zeta(B_{1/2})}\le C,
\qquad [D^2u]_{C^\zeta(B_{R/2})}\le CR^{-\zeta}.
\end{equation}
For fixed $x,y\in\R^n$, both points lie in $B_{R/2}$ for all
sufficiently large $R$. The second estimate then gives
\[
|D^2u(x)-D^2u(y)|\le CR^{-\zeta}|x-y|^\zeta.
\]
Letting $R\to\infty$ makes the right-hand side zero.
Thus $D^2u(x)=D^2u(y)$ for every pair of points.

Write this constant Hessian as $A$. Integration gives
$u(x)=\tfrac12x^TAx+\ell\cdot x+c$, and the equation gives
$\sigma_k(A)=1$.

Finally, take $x=re$ in \eqref{eq:liouville-growth}, where $|e|=1$.
Substitution of the quadratic expression and division by $r^2$ give
\[
\frac12e^TAe+\frac{\ell\cdot e}{r}+\frac c{r^2}
\ge b-\frac B{r^2}.
\]
Letting $r\to\infty$ removes the affine and constant terms, yielding
\begin{equation}\label{eq:matrix-lower}
\tfrac12e^TAe\ge b.
\end{equation}
Hence $A\ge2bI$, proving Theorem~\ref{thm:main}.
\qed

\section{Viscosity solutions}\label{sec:viscosity}

We specify the viscosity formulation and prove automatic admissibility.
Approximation by smooth solutions then gives the interior estimate in
Theorem~\ref{thm:viscosity-interior} and the Liouville theorem in
Theorem~\ref{thm:viscosity-liouville}.

\subsection{The viscosity formulation}\label{sec:viscosity-formulation}

Let $\Omega\subset\R^n$ be open. A proper upper semicontinuous function
$u:\Omega\to[-\infty,\infty)$ is $k$-convex if every $C^2$ function
touching $u$ from above has Hessian in $\clG_k$ at the contact point.

Here proper means that $u$ is not identically $-\infty$ on any component.
We denote this class by $\Phi_k(\Omega)$ and use its upper semicontinuous
representatives throughout. For weak functions, the admissibility
condition means membership in $\Phi_k(\Omega)$.

A continuous function $u$ is a viscosity solution of
$\sigma_k(D^2u)=1$ if every $C^2$ upper touching test $\phi$ satisfies
\begin{equation}\label{eq:viscosity-upper}
\sigma_k(D^2\phi(x))\ge1,
\end{equation}
and every $C^2$ lower touching test $\phi$ with $D^2\phi(x)\in\clG_k$ satisfies
\begin{equation}\label{eq:viscosity-lower}
\sigma_k(D^2\phi(x))\le1.
\end{equation}
Touching from above means that $u-\phi$ has a local maximum zero at
the contact point; touching from below means a local minimum zero.
The lower tests are taken on the elliptic branch.

The following lemma shows that the upper-test inequality already forces
admissibility, giving the usual formulation in \cite[Section 2]{TWII}.

\begin{lemma}[Automatic viscosity admissibility]\label{lem:viscosity-admissibility}
Let $u\in C(\Omega)$ satisfy \eqref{eq:viscosity-upper} for every
$C^2$ upper touching test. Then every such test has Hessian in $\G_k$
at its contact point. In particular, $u\in\Phi_k(\Omega)$.
\end{lemma}

\begin{proof}
Fix an upper touching test $\phi$ at $x_0$, and put $A=D^2\phi(x_0)$.
Choose $T>0$ such that $A+TI>0$, and define
\begin{equation}\label{eq:viscosity-continuity-set}
G=\{t\in[0,T]:A+tI\in\G_k\}.
\end{equation}
The matrix $A+TI$ is positive definite, so $T\in G$.
For $t_0\in G$, continuity keeps all the finitely many inequalities
$\sigma_j(A+tI)>0$ valid when $t$ is sufficiently close to $t_0$.
Thus $G$ is relatively open in $[0,T]$.

For every $t\in[0,T]$, introduce
\[
\phi_t(x)=\phi(x)+\tfrac t2|x-x_0|^2.
\]
The added term is nonnegative and vanishes at $x_0$, so $\phi_t$
still touches $u$ from above. Applying \eqref{eq:viscosity-upper} to
this test gives $\sigma_k(A+tI)\ge1$ for the whole interval.

For $t\in G$, Maclaurin's inequality \eqref{eq:maclaurin} yields
\begin{equation}\label{eq:viscosity-maclaurin}
\sigma_j(A+tI)\ge
\binom{n}{j}\binom{n}{k}^{-j/k}>0,
\qquad 1\le j<k.
\end{equation}
To prove closedness, take $t_m\in G$ with $t_m\to t\in[0,T]$.
Continuity of each symmetric polynomial preserves the positive
lower bounds in \eqref{eq:viscosity-maclaurin}.

Together with $\sigma_k(A+tI)\ge1$, these bounds imply
$A+tI\in\G_k$. Hence $t\in G$, and $G$ is relatively closed.

Connectedness gives $G=[0,T]$, so $A\in\G_k$.
Since the upper touching test was arbitrary, $u\in\Phi_k(\Omega)$.
\end{proof}

This argument does not use the growth condition.
For the classical theorems, admissibility follows instead from the
growth condition and continuity of the Hessian, as in Section~\ref{sec:branch}.

\subsection{A margin in the smooth estimate}

We need a small margin in the growth condition to accommodate the error
in approximation. The proof of Theorem~\ref{thm:interior} gives the following
variant with constants depending on the same parameters.

\begin{lemma}\label{lem:margin}
Let $u\in C^\infty(B_R)$ satisfy $\sigma_k(D^2u)=1$ and $u(0)=0$.
Suppose $b>0$ and
\begin{equation}\label{eq:margin-growth}
u(x)\ge b|x|^2-B,\qquad 0\le B\le\frac{bR^2}{2}.
\end{equation}
There are constants $\eta=\eta(n,k,b)\in(0,1)$ and $C=C(n,k,b)$ such that
\begin{equation}\label{eq:margin-estimate}
\sup_{B_{\eta R}}\left(\frac{|u|}{R^2}+
\frac{|Du|}{R}+|D^2u|\right)\le C.
\end{equation}
\end{lemma}

\begin{proof}
Rescale to $R=1$ and choose $r_0\in(1/\sqrt2,1)$.
For $r_0\le|x|<1$, the lower bound gives
\[
u(x)\ge br_0^2-B\ge b(r_0^2-1/2)>0=u(0).
\]
Hence a minimum on $\overline{B_{r_0}}$ is an interior minimum in $B_1$.
The continuity argument in Section~\ref{sec:branch} then gives
$D^2u\in\G_k$ throughout $B_1$.

Choose a regular value $t\in(b/8,b/4)$, and let $\Omega_t$ be the
component of $\{u<t\}$ containing the origin. Then
\begin{equation}\label{eq:margin-sublevel}
\Omega_t\subset B_{\sqrt{(t+B)/b}}\subset B_{\sqrt3/2},\qquad
\diam\Omega_t\le\sqrt3.
\end{equation}
As in Section~\ref{sec:interior}, its boundary is $(k-1)$-convex.
For $U=u-t$ and $d_t(x)=\dist(x,\partial\Omega_t)$,
Proposition~\ref{prop:boundary} gives
\begin{equation}\label{eq:margin-depth}
\frac b8<t=-U(0)\le C_0\,3^{1-\beta/2}d_t(0)^\beta.
\end{equation}
Choose
\[
\gamma_0=\min\left\{\frac14,
\left(\frac{b}{8C_0\,3^{1-\beta/2}}\right)^{1/\beta}\right\},
\qquad \eta=\frac{\gamma_0}{2}.
\]
Then \eqref{eq:margin-depth} gives $d_t(0)\ge\gamma_0$.
The segment argument following \eqref{eq:rho-definition} gives
$B_{\gamma_0}\subset\Omega_t$.

Thus $B_\eta\subset\Omega_t$. Since distance is $1$-Lipschitz and
$\eta=\gamma_0/2$, every $x\in B_\eta$ satisfies $d_t(x)\ge\eta$.

Writing $D_t=\diam\Omega_t\le\sqrt3$, we therefore have
$d_t(x)\ge(\eta/\sqrt3)D_t$ on $B_\eta$.
Lemma~\ref{lem:domain-interior} applies with $\delta=\eta/\sqrt3$.

Finally, $u=U+t$ and $t<b/4$, so the resulting bounds for $U$, $DU$,
and $D^2U$ give \eqref{eq:margin-estimate} when $R=1$.
The change of variables $x=Ry$ restores the stated powers of $R$.
\end{proof}

\subsection{Approximation by solutions of the equation}

\begin{proof}[Proof of Theorem~\ref{thm:viscosity-interior}]
By Lemma~\ref{lem:viscosity-admissibility}, the continuous viscosity
solution $u$ is admissible.

Rescale to $R=1$ and set $r=\sqrt3/2$.
The smaller ball has closure inside $B_1$, so $u$ is continuous on
$\overline{B_r}$ and its boundary values can be approximated uniformly.

Choose smooth functions $\varphi_m$ on $\partial B_r$ with
\begin{equation}\label{eq:weak-boundary-approximation}
\varepsilon_m:=\|\varphi_m-u\|_{L^\infty(\partial B_r)}\longrightarrow0.
\end{equation}
The sphere $\partial B_r$ is connected, and its principal curvatures
are all $1/r$. Consequently,
\[
\sigma_{k-1}(\kappa)=\binom{n-1}{k-1}r^{1-k}>0.
\]
The smooth Dirichlet theorem of Caffarelli, Nirenberg, and Spruck
\cite[Theorem 3]{CNS} therefore supplies admissible solutions, smooth
up to the boundary, satisfying
\begin{equation}\label{eq:weak-exact-approximation}
\sigma_k(D^2u_m)=1\quad\text{in }B_r,\qquad
u_m=\varphi_m\quad\text{on }\partial B_r.
\end{equation}
Maclaurin's inequality \eqref{eq:maclaurin} bounds every lower symmetric
polynomial of $D^2u_m$ away from zero. Continuity preserves these bounds
on $\partial B_r$.

For fixed $m$, it follows that $D^2u_m(\overline{B_r})\Subset\G_k$.
Thus there exists $\delta_m>0$ such that perturbations by $\pm2\delta I$
preserve admissibility whenever $0<\delta<\delta_m$.

We explain the comparison that transfers the boundary error into the ball.
Define
\[
q_m^+=u_m+\varepsilon_m-\delta(|x|^2-r^2),
\qquad q_m^-=u_m-\varepsilon_m+\delta(|x|^2-r^2).
\]
Take $0<\delta<\delta_m$. For any symmetric matrix $A$,
\[
\frac{d}{dt}\sigma_k(A+tI)
=(n-k+1)\sigma_{k-1}(A+tI)>0
\quad\text{when }A+tI\in\G_k.
\]
Apply this identity on the admissible segments from $D^2u_m$ to
$D^2u_m\pm2\delta I$. Since $\sigma_k(D^2u_m)=1$, it gives
$\sigma_k(D^2q_m^+)<1<\sigma_k(D^2q_m^-)$.

The boundary approximation gives $q_m^-\le u\le q_m^+$ on
$\partial B_r$. If $u-q_m^+$ had a positive interior maximum,
a constant translate of $q_m^+$ would violate the upper-test inequality.

Similarly, a positive interior maximum of $q_m^--u$ would provide an
admissible lower touching test with $\sigma_k>1$, contradicting
\eqref{eq:viscosity-lower}. Hence $q_m^-\le u\le q_m^+$ in the ball.

Letting $\delta\downarrow0$ proves
\begin{equation}\label{eq:weak-uniform-convergence}
\|u_m-u\|_{L^\infty(B_r)}\le\varepsilon_m.
\end{equation}

Normalize at the center by putting $v_m=u_m-u_m(0)$.
The equation is unchanged by subtraction of this constant.
Since $u(0)=0$, we have $|u_m(0)|\le\varepsilon_m$, and hence
$\|v_m-u\|_{L^\infty(B_r)}\le2\varepsilon_m$.

Combining this estimate with the lower bound for $u$, we obtain
\begin{equation}\label{eq:weak-growth-error}
v_m(x)\ge b|x|^2-B-2\varepsilon_m,\qquad
B+2\varepsilon_m\le\frac{br^2}{2}=\frac{3b}{8}
\end{equation}
for all sufficiently large $m$: the assumption $B\le b/4$ leaves
a margin $3b/8-b/4=b/8$, which absorbs $2\varepsilon_m$.
Lemma~\ref{lem:margin}, applied on $B_r$, now gives
\begin{equation}\label{eq:weak-uniform-c2}
\sup_{B_{\eta r}}\bigl(|v_m|+|Dv_m|+|D^2v_m|\bigr)\le C(n,k,b).
\end{equation}

The Hessians lie in a common compact subset of $\G_k$, by
\eqref{eq:maclaurin} and $\sigma_k(D^2v_m)=1$.
The construction in Section~\ref{sec:liouville} gives a common concave,
uniformly elliptic extension of $\sigma_k^{1/k}$ on this Hessian range.

The Evans--Krylov estimate gives a uniform $C^{2,\zeta}$ bound on
$B_{\eta r/2}$ for some $\zeta>0$. Compactness gives a subsequence
converging in $C^2$ on every compact subset of this ball.

Its limit is $u$, because $v_m\to u$ uniformly.
Continuity of $\sigma_k$ therefore gives $\sigma_k(D^2u)=1$
classically on $B_{\eta r/2}$.

Pass to the limit at each point of this smaller ball in
\eqref{eq:weak-uniform-c2}, and then take the supremum. This gives
\[
\sup_{B_{\eta r/2}}\bigl(|u|+|Du|+|D^2u|\bigr)\le C(n,k,b).
\]

For smoothness, differentiate the equation of the smooth approximants.
With $F=\sigma_k^{1/k}$, each coordinate derivative satisfies
\[
\sum_{i,j}F^{ij}(D^2v_m)\,\partial_{ij}(\partial_\ell v_m)=0,
\qquad 1\le\ell\le n.
\]
The coefficients have uniform positive ellipticity bounds and uniform
$C^{0,\zeta}$ norms on smaller balls. Interior Schauder estimates
therefore give uniform $C^{3,\zeta}$ bounds for $v_m$ there.

Iterating on compactly contained balls gives bounds of every order,
so $u\in C^\infty(B_{\eta r/2})$. Rescaling gives
\eqref{eq:viscosity-interior-estimate} with $\rho=\eta r/2$.
\end{proof}

\subsection{The Liouville theorem}

\begin{proof}[Proof of Theorem~\ref{thm:viscosity-liouville}]
Let $u$ be the given entire viscosity solution. Normalize by
$v=u-u(0)$ and put $\widetilde B=B+u(0)\ge0$.
For all sufficiently large $R$, we have
\begin{equation}\label{eq:weak-global-growth}
v(x)\ge b|x|^2-\widetilde B,\qquad
0\le\widetilde B\le\frac{bR^2}{4}.
\end{equation}
Theorem~\ref{thm:viscosity-interior} gives
$u\in C^\infty(B_{\rho R})$, where $\rho>0$ depends only on $n,k,b$.
These balls exhaust $\R^n$, so $u\in C^\infty(\R^n)$.

The quadratic lower bound is unchanged.
Theorem~\ref{thm:main} therefore gives \eqref{eq:viscosity-conclusion}.
\end{proof}

\section{Admissible weak solutions}\label{sec:ibp}

We use Trudinger's definition of admissible weak solutions by approximation.
His viscosity characterization gives the two weak-solution theorems.
We then explain the integration-by-parts identities satisfied by these solutions.

\subsection{Trudinger's definition}\label{sec:ibp-formulation}

Following \cite[p.~1252, (1.4)]{Tr97}, a function $u\in C(\Omega)$ is an
admissible weak solution of $\sigma_k(D^2u)=1$ if there exist
$u_m\in C^2(\Omega)$ satisfying
\begin{equation}\label{eq:weak-approximation}
\begin{aligned}
D^2u_m&\in\clG_k\quad\text{in }\Omega,\qquad
u_m\to u\quad\text{locally uniformly},\\
\sigma_k(D^2u_m)&\to1\quad\text{in }L^1_{\mathrm{loc}}(\Omega).
\end{aligned}
\end{equation}
The approximants are $k$-admissible in the weak sense, including points
where their Hessians lie on the boundary of $\G_k$.
Their locally uniform limit belongs to $\Phi_k(\Omega)$.

\subsection{The Liouville theorem and the interior estimate}

For a continuous right-hand side, Trudinger characterizes admissible
weak solutions by the viscosity inequalities \cite[p.~1255]{Tr97}.
We apply this characterization with right-hand side equal to one.

\begin{proof}[Proof of Theorem~\ref{thm:ibp-liouville}]
The function $u$ is continuous by definition.
Trudinger's characterization makes it a viscosity solution of
$\sigma_k(D^2u)=1$ throughout $\R^n$.

The lower bound is unchanged, so all hypotheses of
Theorem~\ref{thm:viscosity-liouville} hold.
That theorem gives \eqref{eq:ibp-conclusion}.
\end{proof}

\begin{proof}[Proof of Theorem~\ref{thm:ibp-interior}]
Trudinger's characterization makes $u$ a continuous viscosity solution
on $B_R$. The value $u(0)=0$ and the lower bound with
$0\le B\le bR^2/4$ are unchanged.

Apply Theorem~\ref{thm:viscosity-interior} to obtain
$u\in C^\infty(B_{\rho R})$ and \eqref{eq:ibp-interior-estimate}.
Its constants depend only on $n,k,b$.
\end{proof}

\subsection{Integration by parts and Hessian measures}

For smooth $u$, write $T_{k-1}^{ij}(D^2u)$ for the entries of the Newton
tensor $D\sigma_k(D^2u)$, as in Section~\ref{sec:interior}.
Its divergence-free property and the homogeneity of $\sigma_k$ give
\begin{equation}\label{eq:newton-divergence}
\sum_i\partial_i T_{k-1}^{ij}(D^2u)=0,\qquad
\sum_{i,j}T_{k-1}^{ij}(D^2u)u_{ij}=k\sigma_k(D^2u).
\end{equation}
The second identity follows by differentiating
$\sigma_k(tH)=t^k\sigma_k(H)$ at $t=1$. The first identity removes the
derivative of the Newton tensor in the product rule, giving
\[
\sum_{i,j}\partial_i\bigl(T_{k-1}^{ij}(D^2u)u_j\bigr)
=k\sigma_k(D^2u).
\]

Multiply by $\phi\in C_c^\infty(\Omega)$ and integrate first in $x_i$.
Then integrate in $x_j$; symmetry of the Newton tensor gives the
corresponding divergence identity in that index. Thus
\begin{equation}\label{eq:newton-ibp}
\begin{aligned}
\int_\Omega\phi\,\sigma_k(D^2u)\,dx
&=-\frac1k\sum_{i,j}\int_\Omega \phi_i T_{k-1}^{ij}(D^2u)u_j\,dx\\
&=\frac1k\sum_{i,j}\int_\Omega u T_{k-1}^{ij}(D^2u)\phi_{ij}\,dx.
\end{aligned}
\end{equation}
Both boundary terms vanish because $\phi$ has compact support.
The divergence identities and their mixed versions are discussed in
\cite[Section 2]{TWIII}.

The identity extends to $C^2$ functions by local smooth approximation.
Apply it to each $u_m$ in \eqref{eq:weak-approximation}.
The convergence of the Hessian densities in $L^1_{\mathrm{loc}}$ yields
\begin{equation}\label{eq:weak-ibp}
\frac1k\lim_{m\to\infty}
\sum_{i,j}\int_\Omega
u_m T_{k-1}^{ij}(D^2u_m)\phi_{ij}\,dx
=\int_\Omega\phi\,dx,
\qquad \phi\in C_c^\infty(\Omega).
\end{equation}
Only the contracted integral is passed to the limit.
No pointwise second derivatives of $u$ or separate limits of the
products $u_mT_{k-1}^{ij}(D^2u_m)$ are needed.

For $k=2$, the Newton tensor has the explicit form
$T_1^{ij}(D^2v)=\Delta v\,\delta_{ij}-v_{ij}$.
For smooth $v$, integrate the two terms separately:
\[
\begin{aligned}
\int v\Delta v\,\Delta\phi
&=-\int |Dv|^2\Delta\phi-\int vDv\cdot D\Delta\phi,\\
-\sum_{i,j}\int vv_{ij}\phi_{ij}
&=\sum_{i,j}\int v_i v_j\phi_{ij}+\int vDv\cdot D\Delta\phi.
\end{aligned}
\]
The last terms cancel. Dividing the resulting identity by two gives
\begin{equation}\label{eq:ibp-two-gradient}
\frac12\sum_{i,j}\int vT_1^{ij}(D^2v)\phi_{ij}
=\frac12\sum_{i,j}\int v_i v_j\phi_{ij}
-\frac12\int|Dv|^2\Delta\phi.
\end{equation}
If $v\in W^{1,2}_{\mathrm{loc}}$, ordinary mollifications converge strongly
in this space. Their quadratic gradient products consequently converge
in $L^1_{\mathrm{loc}}$, by the Cauchy--Schwarz inequality.

Thus the right-hand side is the pairing
$\langle\mathcal S_2[v],\phi\rangle$ used in
Remark~\ref{rem:weak-counterexample}.

We also use the Hessian measure associated with
$u\in\Phi_k(\Omega)$ by Trudinger and Wang \cite[Theorem 1.1]{TWII}.
Its classical density and weak continuity are
\begin{equation}\label{eq:hessian-measure-definition}
\begin{aligned}
\mu_k[u]&=\sigma_k(D^2u)\,dx\quad\text{if }u\in C^2(\Omega),\\
\mu_k[u_m]&\rightharpoonup\mu_k[u]
\quad\text{if }u_m\to u\text{ in }L^1_{\mathrm{loc}}(\Omega).
\end{aligned}
\end{equation}
Here all functions in the convergence statement belong to $\Phi_k(\Omega)$.
Weak convergence is tested against continuous compactly supported functions.

For the approximants in \eqref{eq:weak-approximation}, the densities also
converge to one in $L^1_{\mathrm{loc}}$. Hence every admissible weak
solution in Trudinger's sense satisfies
\[
\mu_k[u]=dx.
\]

\subsection{Why almost-everywhere equality is insufficient}

The necessary condition $\mu_k[u]=dx$ detects singular Hessian mass
that is invisible in an almost-everywhere equation. For example, set
\begin{equation}\label{eq:weak-ae-example}
c=\binom{n}{k}^{-1/k},\qquad
w(x)=\frac c2|x|^2+|x_1|.
\end{equation}
This function is convex, has positive quadratic lower growth, and
satisfies $\sigma_k(D^2w)=1$ almost everywhere. It is not a quadratic
polynomial.

Indeed, on either half-space $\{x_1>0\}$ or $\{x_1<0\}$, the term
$|x_1|$ is affine. Thus $D^2w=cI$ there and
$\sigma_k(D^2w)=\binom nk c^k=1$.

Its Hessian measure contains a nonzero singular part:
\begin{equation}\label{eq:weak-ae-measure}
\mu_k[w]=dx+
2\binom{n-1}{k-1}c^{k-1}\,
\mathcal H^{n-1}\!\restriction_{\{x_1=0\}}.
\end{equation}
To compute it, let $\chi_\varepsilon$ be a one-dimensional mollification
of $|\cdot|$ and set $w_\varepsilon=c|x|^2/2+\chi_\varepsilon(x_1)$.
Then
\[
\begin{aligned}
D^2w_\varepsilon&=cI+\chi_\varepsilon''(x_1)e_1\otimes e_1,\\
\sigma_k(D^2w_\varepsilon)
&=1+\binom{n-1}{k-1}c^{k-1}\chi_\varepsilon''(x_1).
\end{aligned}
\]

Since $(|t|)''=2\delta_0$ in distributions,
$\chi_\varepsilon''(t)\,dt\rightharpoonup2\delta_0$.
Integration in the remaining coordinates gives the hypersurface measure
in \eqref{eq:weak-ae-measure}.

Weak continuity identifies this limit with $\mu_k[w]$.
Since the hypersurface term is nonzero, $\mu_k[w]\ne dx$.
Thus $w$ is not an admissible weak solution in Trudinger's sense.

\appendix
\titleformat{\section}{\normalfont\Large\bfseries}{Appendix~\thesection}{1em}{}
\section{Two algebraic lemmas}\label{sec:algebra}

We prove the two algebraic inequalities used to construct the distance barrier.
The first controls a quadratic symmetric polynomial under a reciprocal
constraint.

The second transfers this inequality to the fractional transformation
of the principal curvatures. Its conclusion gives the admissibility
of the distance power in Section~\ref{sec:distance}.

\begin{lemma}[A reciprocal inequality]\label{lem:quadratic}
Let $N\ge1$ and $\eta\in\R^N$ satisfy
\begin{equation}\label{eq:eta-hypotheses}
\eta_i>-1\quad(1\le i\le N),\qquad
\sum_{i=1}^N\frac{\eta_i}{1+\eta_i}\ge0.
\end{equation}
Then
\begin{equation}\label{eq:quadratic}
\sigma_2(\eta)+\left(1-\frac1{2N}\right)\sigma_1(\eta)\ge0.
\end{equation}
\end{lemma}

\begin{proof}
When $N=1$, the denominator $1+\eta_1$ is positive, so the assumption
gives $\eta_1\ge0$. Since $\sigma_2(\eta)=0$ in one variable, the
conclusion follows.

Assume $N\ge2$. We compare the sum of squares with the square of the
sum. Set $S=\sum_i\eta_i$ and $E=\sum_i\eta_i^2/(1+\eta_i)$.
The identities
\[
\frac{\eta_i}{1+\eta_i}=1-\frac1{1+\eta_i},\qquad
\frac{\eta_i^2}{1+\eta_i}=\eta_i-\frac{\eta_i}{1+\eta_i}
\]
convert the hypothesis into
\begin{equation}\label{eq:reciprocal}
\sum_i\frac1{1+\eta_i}\le N,\qquad
0\le E=S-\sum_i\frac{\eta_i}{1+\eta_i}\le S.
\end{equation}

In particular, $S\ge0$. Fix an index $i$ and apply Cauchy--Schwarz
to the $N-1$ positive numbers $1+\eta_j$, with $j\ne i$:
\[
(N-1)^2\le
\left(\sum_{j\ne i}(1+\eta_j)\right)
\left(\sum_{j\ne i}\frac1{1+\eta_j}\right).
\]
The second factor is at most $N$ by \eqref{eq:reciprocal}. Dividing
by this positive factor gives
\begin{equation}\label{eq:other-entries}
\sum_{j\ne i}(1+\eta_j)
\ge\frac{(N-1)^2}{\displaystyle\sum_{j\ne i}(1+\eta_j)^{-1}}
\ge\frac{(N-1)^2}{N}.
\end{equation}
Subtracting from $\sum_j(1+\eta_j)=N+S$, we find
$1+\eta_i\le S+2-N^{-1}$ for every $i$.

Multiply this bound by the nonnegative number
$\eta_i^2/(1+\eta_i)$ and sum over $i$.
Using $E\le S$ and $S+2-N^{-1}>0$, we obtain
\begin{equation}\label{eq:square-bound}
\sum_i\eta_i^2
\le\left(S+2-\frac1N\right)E
\le S^2+\left(2-\frac1N\right)S.
\end{equation}
Finally, expand $S^2$ into its diagonal and off-diagonal terms:
\[
2\sigma_2(\eta)=S^2-\sum_i\eta_i^2
\ge-\left(2-\frac1N\right)S.
\]
Dividing by two and using $S=\sigma_1(\eta)$ proves
\eqref{eq:quadratic}.
\end{proof}

\begin{lemma}[A fractional transformation of the symmetric cone]\label{lem:transform}
Let $2\le k\le n$, $a\in\G_{k-1}^{(n-1)}$, and $a_i<1$ for every $i$.
Define
\begin{equation}\label{eq:algebra-constants}
\mu_i=\frac{a_i}{1-a_i},\qquad N=n-k+1,\qquad
\theta=\frac2k\left(1-\frac1{2N}\right).
\end{equation}
Then $0<\theta<1$, $\mu\in\G_{k-1}^{(n-1)}$, and
\begin{equation}\label{eq:transform}
\sigma_k(\mu)+\theta\sigma_{k-1}(\mu)\ge0.
\end{equation}
In particular, $(\mu_1,\ldots,\mu_{n-1},\theta)\in\clG_k^{(n)}$.
\end{lemma}

\begin{proof}
We connect the transformed vector $\mu$ to the original vector $a$
by introducing
\begin{equation}\label{eq:flow}
\mu_i(s)=\frac{\mu_i}{1+s\mu_i}
=\frac{a_i}{1-(1-s)a_i},\qquad 0\le s\le1.
\end{equation}
Here $\mu(0)=\mu$ and $\mu(1)=a$. The denominator in the second
expression is $(1-s)(1-a_i)+s$, which is positive for $0\le s\le1$.
Moreover,
\begin{equation}\label{eq:monotone-flow}
\mu_i(s)-a_i
=\frac{(1-s)a_i^2}{1-(1-s)a_i}\ge0.
\end{equation}
Thus $\mu(s)$ is obtained from $a\in\G_{k-1}$ by adding a vector
with nonnegative entries. Monotonicity of the cone gives
$\mu(s)\in\G_{k-1}$, including $\mu=\mu(0)$.

Now introduce the polynomials
\begin{equation}\label{eq:polynomials}
P(s)=\prod_{i=1}^{n-1}(1+s\mu_i),\qquad Q(s)=P^{(k-2)}(s).
\end{equation}
In a nonzero term of $P^{(j)}$, the derivatives fall on $j$ distinct
linear factors. Each index set $I$ with $j$ elements occurs $j!$ times;
dividing by $P(s)$ leaves $\prod_{i\in I}\mu_i/(1+s\mu_i)$.

Summing over all choices of $j$ indices therefore gives
\begin{equation}\label{eq:derivative-product}
P^{(j)}(s)=j!P(s)\sigma_j(\mu(s)),\qquad 0\le j\le k.
\end{equation}
Since $1+\mu_i=(1-a_i)^{-1}>0$, all factors of $P$ are positive on
$[0,1]$. The cone membership just proved gives
\[
Q=(k-2)!P\sigma_{k-2}(\mu(s))>0,\qquad
Q'=(k-1)!P\sigma_{k-1}(\mu(s))>0.
\]
For $k=2$, the first identity uses the convention $\sigma_0=1$.

The polynomial $P$ has only real roots. Repeated application of
Rolle's theorem, counting multiplicities, shows that $Q$ also has
only real roots.

Since $Q'>0$ on $[0,1]$, the polynomial $Q$ is nonconstant.
In particular, $\deg P\ge k-1$, so differentiation reduces its
degree by exactly $k-2$. Hence
\begin{equation}\label{eq:degree}
\deg Q=\deg P-(k-2)\le n-k+1=N.
\end{equation}
Let $r_1,\ldots,r_{\deg Q}$ be the roots of $Q$, counted with
multiplicity. Since $Q>0$ on $[0,1]$, each root lies in
$(-\infty,0)\cup(1,\infty)$.

For $\eta_j=-1/r_j$, a negative root gives $\eta_j>0$, while a root
greater than one gives $-1<\eta_j<0$.
Factoring $Q$ and normalizing its value at zero yields
\begin{equation}\label{eq:factorization}
\frac{Q(s)}{Q(0)}=\prod_{j=1}^N(1+s\eta_j),\qquad \eta_j>-1.
\end{equation}
If $\deg Q<N$, we add zero entries to $\eta$.

Logarithmic differentiation at $s=1$ yields
\begin{equation}\label{eq:log-derivative}
\sum_{j=1}^N\frac{\eta_j}{1+\eta_j}
=\frac{Q'(1)}{Q(1)}>0.
\end{equation}
Thus $\eta$ satisfies both hypotheses of Lemma~\ref{lem:quadratic}.
To express its conclusion in terms of $\mu$, differentiate
\eqref{eq:factorization} at zero:
\[
\sigma_1(\eta)=\frac{Q'(0)}{Q(0)},\qquad
\sigma_2(\eta)=\frac{Q''(0)}{2Q(0)}.
\]
Since $P(0)=1$, formula \eqref{eq:derivative-product} evaluates these
three derivatives. We obtain
\begin{equation}\label{eq:coefficients}
\begin{aligned}
\sigma_1(\eta)&=\frac{(k-1)\sigma_{k-1}(\mu)}{\sigma_{k-2}(\mu)},\\
\sigma_2(\eta)&=\frac{k(k-1)\sigma_k(\mu)}{2\sigma_{k-2}(\mu)}.
\end{aligned}
\end{equation}
The denominator $\sigma_{k-2}(\mu)$ is positive. Substitute these
identities into \eqref{eq:quadratic} and multiply by
$2\sigma_{k-2}(\mu)/(k(k-1))$ to obtain
\[
\sigma_k(\mu)+\frac2k\left(1-\frac1{2N}\right)
\sigma_{k-1}(\mu)\ge0.
\]
This is \eqref{eq:transform}. When $k=n$, the second identity in
\eqref{eq:coefficients} has both sides zero, since $\deg Q\le1$ and
$\sigma_n(\mu)=0$.

Since $\mu\in\G_{k-1}$ and $\theta>0$, for $1\le j<k$ we have
\begin{equation}\label{eq:lower-symmetric}
\sigma_j(\mu,\theta)=\sigma_j(\mu)+\theta\sigma_{j-1}(\mu)>0.
\end{equation}
For every $\eps>0$, \eqref{eq:transform} also gives
\begin{equation}\label{eq:closure}
\sigma_k(\mu,\theta+\eps)\ge\eps\sigma_{k-1}(\mu)>0.
\end{equation}
Thus $(\mu,\theta+\eps)\in\G_k$. Letting $\eps\downarrow0$ proves
the closure statement. The formula for $\theta$ gives $0<\theta<1$.
\end{proof}

\section*{Acknowledgements}
The research is assisted by ChatGPT 6.

\medskip
\noindent\textit{Jingchen Hu:}
\href{mailto:jingchenhoo@gmail.com}{jingchenhoo@gmail.com},
\href{mailto:jingchenhu@scu.edu.cn}{jingchenhu@scu.edu.cn}.
\par\noindent\textit{Heming Jiao:}
\href{mailto:jiao@hit.edu.cn}{jiao@hit.edu.cn}.
\end{document}